\documentclass[12 pt]{amsart}
\usepackage[cp1251]{inputenc}
\usepackage{amscd,amsfonts,amssymb,amsmath}
\usepackage{tikz-cd}
\usepackage[all]{xy}%диаграммы
\usepackage{xcolor}
\usepackage{placeins}
\usepackage{hyperref}

\newtheorem{theorem}{Theorem}[section]
\theoremstyle{definition}

\newtheorem{corollary}[theorem]{Corollary}
\newtheorem{lemma}[theorem]{Lemma}
\newtheorem{proposition}[theorem]{Proposition}
\newtheorem{question}[theorem]{Question}

\newtheorem{definition}[theorem]{Definition}
\newtheorem{example}[theorem]{Example}

\numberwithin{equation}{section}
\newtheorem*{ack}{Acknowledgement}

\newcommand{\Aut}{\operatorname{Aut}}

\newcommand{\Inn}{\operatorname{Inn}}

\newcommand{\Z}{\mathbb{Z}}

\newcommand{\id}{\mathrm{id}}

 \renewcommand{\phi}{\varphi}
\renewcommand{\epsilon}{\varepsilon}

\newcommand{\set}[1]{\{ #1 \}}

\newcommand{\gen}[1]{\langle #1 \rangle}

\DeclareMathOperator{\Id}{Id}

\makeatletter
\def\thmhead@plain#1#2#3{%
  \thmname{#1}\thmnumber{\@ifnotempty{#1}{ }\@upn{#2}}%
  \thmnote{ {\the\thm@notefont#3}}}
\let\thmhead\thmhead@plain
\makeatother

\makeatletter
\def\@tocline#1#2#3#4#5#6#7{\relax
  \ifnum #1>\c@tocdepth % then omit
  \else
    \par \addpenalty\@secpenalty\addvspace{#2}%
    \begingroup \hyphenpenalty\@M
    \@ifempty{#4}{%
      \@tempdima\csname r@tocindent\number#1\endcsname\relax
    }{%
      \@tempdima#4\relax
    }%
    \parindent\z@ \leftskip#3\relax \advance\leftskip\@tempdima\relax
    \rightskip\@pnumwidth plus4em \parfillskip-\@pnumwidth
    #5\leavevmode\hskip-\@tempdima
      \ifcase #1
       \or\or \hskip 1em \or \hskip 2em \else \hskip 3em \fi%
      #6\nobreak\relax
    \hfill\hbox to\@pnumwidth{\@tocpagenum{#7}}\par% <---- \dotfill -> \hfill
    \nobreak
    \endgroup
  \fi}
\makeatother

\title{Some properties of relative Rota--Baxter operators on groups}
\author[V.~G.~Bardakov, T.~A.~ Kozlovskaya, P.~ P. Sokolov, K.~ V. ~Zimireva, M.~ N. ~Zonov]{Valeriy G.~Bardakov, Tatyana A.~ Kozlovskaya, Pavel P. Sokolov, \\Ksenia V. ~Zimireva, Matvei N. ~Zonov}

\date{\today}

\begin{document}
\sloppy
%\hfill{17B38, 20D08 (MSC2020)}

\begin{abstract} We find connection between relative Rota--Baxter operators and usual Rota--Baxter operators. We prove that any relative Rota--Baxter operator on a group $H$ with respect to  $(G, \Psi)$ defines a Rota--Baxter operator on the semi-direct product $H\rtimes_{\Psi} G$. On the other side, we give condition under which a Rota--Baxter operator on the semi-direct product $H\rtimes_{\Psi} G$ defines a relative  Rota--Baxter operator on  $H$ with respect to  $(G, \Psi)$.
We introduce homomorphic post-groups and    find  their connection   with  $\lambda$-homomorphic skew left braces.
Further, we construct post-group on arbitrary group and a family post-groups which depends on integer parameter on any two-step nilpotent group. 
We find all verbal solutions of the 
quantum Yang-Baxter equation on two-step nilpotent group. \\

\textit{Keywords: Group, nilpotent group, semi-direct product, skew brace,  Rota--Baxter operator, relative Rota--Baxter operator, Yang--Baxter equation.} 

 \textit{Mathematics Subject Classification 2020:  17B38,  16T25.} 
\end{abstract}

\maketitle
\tableofcontents

\section{Introduction}

Rota--Baxter operators on  groups were 
introduced in 2021 by L. Guo, H. Lang, Y.~Sheng~\cite{Guo2020}, a  group  with a Rota--Baxter operator  is called a Rota--Baxter group (RB-group).
Properties of RB-groups are actively studied in \cite{BG}. Connection between RB-groups,  the Yang--Baxter equation and skew braces  was found in \cite{BG-1}. 
  The concept of  braces was introduced by Rump \cite{R2007} in 2007 in connection with non-degenerate involutive set theoretic solutions of the quantum Yang--Baxter equation.
    The concept of skew  braces was introduced by Guarnieri and Vendramin \cite{Vendramin} in 2017 in connection with non-involutive non-degenerate set theoretic solutions of the quantum Yang--Baxter equation. 
In \cite{Gon}  the Rota--Baxter operator was defined in cocommutative Hopf algebras. Important example of such algebras give   group rings $\Bbbk[G]$ of group $G$ over a field $\Bbbk$.

In \cite{JSZ} relative Rota--Baxter operators (RRB-operators) and relative Rota--Baxter groups (RRB-groups) were defined. In this case, the relative Rota--Baxter operator depends not only on the group, but also on the  group of its automorphisms.  In the case when this group of automorphisms is a group of inner automorphisms, we obtain the Rota--Baxter operator.  Some properties of relative Rota--Baxter operators are studied in \cite{RS-23}.  
%Some examples of relative Rota--Baxter groups can be found in \cite{RS-23}.
Relative Rota--Baxter operators on an arbitrary Hopf algebra were defined in \cite{BN}.
In \cite{BB} Rota--Baxter and averaging operators on racks and rack algebras were introduced.

In the present paper we  prove some properties of relative Rota--Baxter operators on groups. In particular,  we are studying connections between relative Rota--Baxter operators and usual Rota--Baxter operators. 

The paper is structured as follows. The next section is devoted to review some facts about skew braces, $\lambda$-maps, post-groups, Rota--Baxter operators, relative Rot--Baxter operators, and some connections between them.

In Section \ref{RRB-RB} we find connection between relative Rota--Baxter operators and usual Rota--Baxter operators. We prove (see Proposition \ref{RR-R}) that any relative Rota--Baxter operator on a group $H$ with respect to  $(G, \Psi)$ defines a Rota--Baxter operator on the semi-direct product $H\rtimes_{\Psi} G$. On the another side, Theorem~\ref{R-RR} gives condition under which a Rota--Baxter operator on the semi-direct product $H\rtimes_{\Psi} G$ defines a relative  Rota--Baxter operator on 
$H$ with respect to  $(G, \Psi)$. An example  shows that the construction of this theorem  can map different Rota--Baxter operators to the same relative Rota--Baxter operator. Further,  we introduce relative Rota--Baxter operator of weight $-1$ and find (see Proposition \ref{PM}) connection between relative Rota--Baxter operators of weight 1 and $-1$. Corollary \ref{new} gives a way to construct new RRB-operators from  known RRB-operators.

It is known that if we have a RB-group or a RRB-group $(G, B)$, then we can construct a skew left brace $(G, \cdot, \circ_B)$.  In Section \ref{skew braces} we give example 
(Example \ref{SLB}) which shows that using RRB-operator $B \colon \Z_4 \to \Z_2\times \Z_2$, it is possible to  construct a non-trivial left brace $(\Z_4,+,\circ_B),$ where $(\Z_4,+)\cong \Z_4$ and $(\Z_4,\circ_B)\cong \Z_2\times \Z_2.$ On the other side, any Rota-Baxter left brace on $\Z_4$ is trivial. But Theorem \ref{Zp} shows that we can not take a prime number $p>2$ instead $2$. In this case any RRB-operator $B \colon \Z_{p^2} \to \Z_p\times \Z_p$ is a homomorphism and any relative Rota--Baxter left brace
on $\Z_{p^2}$ is trivial.

In Section \ref{YBE} we introduce homomorphic post-groups and    find (see Proposition~\ref{LSB-PG}) their connection  with  $\lambda$-homomorphic skew left braces.
Further, we construct post-group on arbitrary group and a family of post-groups  on any two-step nilpotent group which depends on an integer parameter. By these post-groups we can construct skew left braces. We show that new operation in these skew left braces defines a two-step nilpotent group, which can  be non isomorphic to the initial two-step nilpotent group. Moreover, it will be proved (in Proposition~\ref{SLB-LSLB}) that  any skew left brace constructed above is a $\lambda$-homomorphic skew left brace. 

As we have already mentioned above, any   skew  left braces  gives 
a non-involutive non-degenerate set theoretic solutions of the quantum Yang-Baxter equation (YBE) \cite{Vendramin}.  In Theorem~\ref{Sol} we find all verbal solutions of the 
quantum Yang-Baxter equation on two-step nilpotent group. 

\bigskip

%%%%%%%%%%%%%%%%%%%%%%%%%%%%%%%%%%%%%%%%%%%%%%%%%%%%%%%%%%%%%%%%%%%%%%%%%%%%

\section{Preliminary results} \label{Prel}

Let $(G,\cdot)$ and $(G,\circ)$ be groups on a same set $G$. The triple $(G,\cdot, \circ)$ is said to be  a \textit{skew left brace} if
$$a\circ(b\cdot c)=(a\circ b)\cdot a ^{-1}\cdot (a\circ c),$$
where $a^{-1}$ is the inverse of $a$ with respect to the operation $\cdot.$ To get a skew right brace we can take another axiom, 
$$(b\cdot c) \circ a =(b\circ a)\cdot a ^{-1}\cdot (c \circ a).$$ A skew left brace which is also a  skew right brace is said to be  a skew two-sided brace.
If $(G,\cdot,\circ)$ is  a skew left brace, then the map $\lambda\colon (G, \circ) \rightarrow \Aut (G,\cdot),$ where $$\lambda_a(b)=a^{-1}\cdot(a\circ b),$$ is called the $\lambda$-\textit{map} of $(G,\cdot,\circ).$
Here we denoted $\lambda_a = \lambda(a)$ the image of the element $a \in G$.

A map of skew left braces $\varphi\colon G\rightarrow H$ is called a homomorphism of skew left braces if $\varphi(a\cdot b)=\varphi(a)\cdot\varphi(b)$ and $\varphi(a\circ b)=\varphi(a)\circ\varphi(b).$

The next definition was introduced in \cite{China_2023}.

\begin{definition} \label{PG}
Let $(G,\cdot)$ be a group and $\triangleright\colon G\times G \rightarrow G$ be a binary  operation on $G$. The algebraic system  $(G,\cdot,\triangleright)$ is called a \textit{post-group}, if the operator $L^{\triangleright}_a\colon G\rightarrow G,$ where $L^{\triangleright}_a(b)=a\triangleright b,$ is an automorphism of the group 
$(G, \cdot)$, and
$$a\triangleright(b\triangleright c) = \bigl(a\cdot(a\triangleright b)\bigr)\triangleright c.$$
\end{definition}

It is easy to see that the axiom of post-group is equivalent to the statement that for any $a, b \in G$ the following holds:
$$
L^{\triangleright}_a L^{\triangleright}_b = L^{\triangleright}_{a L^{\triangleright}_a(b).}
$$
Also remark that the  condition that  $L^{\triangleright}_a$ is an automorphism of $G$ means that
$$
a \triangleright (b \cdot c) = (a \triangleright b) \cdot (a \triangleright c)
$$
for any $a,b,c\in G,$ which is left distributivity of the operation $\triangleright$. Connection between skew left braces and post-groups gives the following theorem.

\begin{theorem}[\textbf{\cite{China_2023}}]\label{th braces <-> post-groups}
1) For any skew left brace $(G,\cdot,\circ)$ let $a\triangleright b := \lambda_a(b).$ Then $(G,\cdot,\triangleright)$ is a post-group.

2) For any post-group $(G,\cdot,\triangleright)$ let $a\circ b = a\cdot(a\triangleright b).$ Then $(G,\cdot,\circ)$ is a skew left brace.
\end{theorem}

%The following definition can.

In \cite{Guo2020} were defined   Rota--Baxter operators of weights $\pm1$ on a group.

\begin{definition}[\cite{Guo2020}] \label{Def:1}
Let $G$ be a group.
\begin{itemize}
\item[(i)] A map $B \colon G\to G$ is called a {\it Rota--Baxter operator  of weight 1 } if
\[
B(g)B(h) = B\left( g B(g) h B(g)^{-1} \right), ~~~g, h\in G;
\]
\item[(ii)] a map $C \colon G\to G$ is called a {\it Rota--Baxter operator  of weight $-1$} if
\[
C(g) C(h) = C\left( C(g) h C(g)^{-1} g \right),~~~g, h\in G.
\]
\end{itemize}
A group endowed with an RB-operator is called a {\it Rota--Baxter group (RB-group)}.
\end{definition}

When we say on a Rota--Baxter operator we say on a Rota--Baxter operator of weight~1.

For two Rota--Baxter groups $(G,B_G)$ and $(H,B_H)$ a group homomorphism $\varphi\colon G\rightarrow H$ is called a \textit{homomorphism of Rota--Baxter groups}, if the following diagram commutes:
$$\xymatrix{G \ar[r]^{B_G} \ar[d]_{\varphi} & G \ar[d]_{\varphi} \\ H \ar[r]^{B_H} & H}$$

It has been shown in \cite{BG-1}, that for a Rota--Baxter group $(G, B)$ and the binary operation $\circ_B,$ defined as 
\begin{equation}\label{eq RBO to skew braces}
a\circ_B b = aB(a)bB(a)^{-1},
\end{equation}
the triple $(G,\cdot,\circ_B)$ is a skew left brace, which is called by {RB skew left brace}.

\begin{proposition}\label{prop RBO to skew braces}
The map from the category of Rota--Baxter groups to the category of skew left braces, given by equation \textnormal{(\ref{eq RBO to skew braces})}, that acts trivially on homomorphisms, is a functor.
\end{proposition} 

\begin{proof}
Let $G$ and $H$ be RotaBaxter groups, and $\phi\colon G \rightarrow H$ a homomorphism of Rota--Baxter groups. For any $a,b\in G$ we have
\begin{gather*}
\phi(a\circ_G b) = \phi(aB_G(a)bB_G(a)^{-1})=\phi(a)\phi B_G(a)\phi(b)\bigl(\phi B_G(a)\bigr)^{-1}= \\
= \phi(a) B_H\bigl(\phi(a)\bigr)\phi(b)B_H\bigl(\phi(a)\bigr)^{-1}=\phi(a)\circ_H \phi(b).
\end{gather*}
\end{proof}

The next definition can be found in \cite{JSZ}.

\begin{definition}
Let $G$ and $H$ be groups, and $\Psi\colon G\rightarrow \Aut H$ is an action of $G$ on $H.$ A map $B\colon H\rightarrow G$ is called a \textit{relative Rota--Baxter operator} (RRB-operator) on $H$ with respect to  $(G, \Psi)$ if 
$$B(h)B(k)=B\bigl(h\Psi_{B(h)}(k)\bigr),  \;\;h, k \in H.$$
The quadruple $(H, G, \Psi, B)$ is called a {\it relative Rota--Baxter group}.
\end{definition}

\begin{example}
If $H = G$ and $\Psi \colon G \to \Inn G$, $\Psi_g = \Psi(g) \colon x \mapsto g x g^{-1}$, $x \in G$, then
$$B(h)B(k)=B\bigl(h B(h) k B(h)^{-1}\bigr),  \;\;h, k \in H,$$
is the usual Rota--Baxter operator on $G,$ and $(G, B)$ is the Rota--Baxter group. 
\end{example}

As in the case of Rota--Baxter groups, the operation
$$
h \circ_B k = h \, \Psi_{B(h)}(k), \;\;h, k \in G,
$$
is a group operation on $G$ (see \cite[Proposition 3.5]{JSZ}).

The next question comes. 

%\begin{question} 
%Is it possible to  prove some analogous of Proposition \ref{prop RBO to skew braces} for RRB-groups?
%\end{question}

\bigskip

%%%%%%%%%%%%%%%%%%%%%%%%%%%%%%%%%%%%%%%%%%%%%

\section{Rota--Baxter and relative Rota--Baxter operators} \label{RRB-RB}

\subsection{Relative Rota--Baxter operators and semi-direct products}

Recall that a semi-direct product $H\rtimes_{\Psi} G$ of groups $G$ and $H$ under the action $\Psi\colon G\rightarrow \Aut H$ is the set of pairs 
$$
H \times G = \{ (h, a)~|~ h \in H, a \in G \}
$$
with multiplication
$$
(h, a) (k, b) = (h \Psi_a(k), a b),~~h, k \in H, a, b \in G.
$$

The following proposition shows that any RRB-operator  defines RB-operator on a semi-direct product. %This proposition and its proof were suggested by V. Gubarev.

\begin{proposition} \label{RR-R}
Let $(H, G, \Psi, B)$ be  a relative Rota--Baxter group. Then the operator 
$$
B' \colon H\rtimes_{\Psi} G \to H\rtimes_{\Psi} G,~~B'((h, a)) = (e, a^{-1} B(h)),~~h \in H, a \in G,
$$
is a Rota--Baxter operator on the semi-direct product $H\rtimes_{\Psi} G$.
\end{proposition}

\begin{proof}
We need to check the equality
$$
B'(u)\,  B'(v) = B' \left(u \,  B'(u) \, v\,  B'(u)^{-1}  \right),~~u, v \in H\rtimes_{\Psi} G.
$$
If $u = (h, a)$, $v = (k, b)$, $h, k \in H$, $a, b \in G$, then the left hand side,
$$
B'(u) \, B'(v) =B'((h, a)) \, B'((k, b)) = (e, a^{-1} \, B(h)) \, (e, b^{-1}\, B(k)) = 
$$
$$
=\left( \Psi_{a^{-1} \, B(h)}(e),  a^{-1} \, B(h) \, b^{-1} \, B(k) \right)=  \left( e,  a^{-1} \, B(h) \, b^{-1} \, B(k) \right).
$$
The right hand side,
$$
B' \left(u \,  B'(u) \, v \, B'(u)^{-1}  \right) = B' \left((h, a)\,  B'((h, a)) \, (k, b) \, B'((h, a))^{-1}  \right) = 
$$
$$
 =  B' \left( (h \, \Psi_a(e), B(h)) \, (k \Psi_b(e), b \, B(h)^{-1} \, a) \right) = B' \left( (h \, \Psi_{B(h)}(k \Psi_b(e)), \, B(h) \,  b \, B(h)^{-1} \, a) \right) 
$$
$$
=  B' \left( (h \, \Psi_{B(h)}(k), \, B(h) \,  b \, B(h)^{-1} \, a) \right) =  \left( e,  \, a^{-1}  B(h) \, b^{-1} \, B(h)^{-1} \, B( h \, \Psi_{B(h)}(k)) \right).
$$
Comparing the left hand side and right hand side, we get
$$
 a^{-1} \, B(h) \, b^{-1} \, B(k) = a^{-1}  B(h) \, b^{-1} \, B(h)^{-1} \, B( h \, \Psi_{B(h)}(k)).
$$
Hence,
$$
 B(h) \,  B(k) =  B( h \, \Psi_{B(h)}(k)),~~h, k \in H.
$$
Since $B$ is a RRB-operator on $H$, this equality holds. 
\end{proof}

We will now present a construction that allows one to build  relative Rota-Baxter operators using  Rota--Baxter operators on semi-direct products.

Let $H\rtimes_{\Psi} G$ be a semi-direct product of groups $H$ and $G$ with respect to some left action $\Psi$ of $G$ on $H.$ Let $B\colon H\rtimes_\Psi G \rightarrow H\rtimes_\Psi G$ be a Rota--Baxter operator.  Consider the projections $\pi_H\colon H\rtimes_{\Psi} G \rightarrow H$ and $\pi_G\colon H\rtimes_{\Psi} G \rightarrow G,$ as well as the restriction
$$B|_H\colon H\rightarrow H\rtimes_{\Psi} G$$
of $B$ to $H.$ Note that $\pi_H$ is not necessarily a group homomorphism and the image of $B|_H$ does not necessary lie in $H$.

The next theorem gives possibilities to construct RRB-operators, using RB-operators on semi-direct products. Constructions of RB-operators on semi-direct and, in particular, on direct products of groups can be found in \cite{BG-1}.

\begin{theorem}\label{R-RR}
Let $B\colon H\rtimes_\Psi G \rightarrow H\rtimes_\Psi G$ be a Rota--Baxter operator. If the image of the map $\pi_H B|_H $ lies in the center of $H,$ then the composition 
$$\pi_G B|_H\colon H\rightarrow G$$ is a relative Rota--Baxter operator with respect to  $(G, \Psi)$.
\end{theorem} 
\begin{proof}
Let $h,k\in H.$ Since $B$ is a Rota-Baxter operator, we have
$$B(h)B(k) = B\bigl(hB(h)kB(h)^{-1}\bigr).$$

Now, express $B$ as a product of $\pi_HB$ and $\pi_GB.$ We have
\begin{equation}\label{eq semidirect rrbo 1}
B(h)B(k) = \pi_HB(h)\pi_GB(h)\pi_HB(k)\pi_GB(k) = \pi_HB(h)\pi_HB(k)^{\bigl(\pi_GB(h)\bigr)^{-1}}\cdot \pi_GB(h)\pi_GB(k)
\end{equation}
and
\begin{equation}\label{eq semidirect rrbo 2}
B(hB(h)kB(h)^{-1}) = \pi_HB\bigl(hB(h)kB(h)^{-1}\bigr)\cdot \pi_GB\bigl(h\pi_HB(h)\pi_GB(h)k\pi_GB(h)^{-1}\pi_HB(h)^{-1}\bigr).
\end{equation}

Since a semi-direct product of groups is a direct product of sets, $a=b$ if and only if $\pi_H(a)=\pi_H(b)$ and $\pi_G(a)=\pi_G(b).$ By applying this reasoning to expressions (\ref{eq semidirect rrbo 1}) and (\ref{eq semidirect rrbo 2}), we obtain:
$$\pi_GB(h)\pi_GB(k) = \pi_GB\bigl(h\pi_HB(h)\pi_GB(h)k\pi_GB(h)^{-1}\pi_HB(h)^{-1}\bigr).$$

Note that $\pi_HB(h)\in Z(H),$ and that $\pi_GB(h)k\pi_GB(h)^{-1}=\Psi_{\pi_GB(h)}(k)\in H.$ We can now simplify:
$$\pi_GB\bigl(h\pi_HB(h)\pi_GB(h)k\pi_GB(h)^{-1}\pi_HB(h)^{-1}\bigr) = \pi_GB\bigl(h\Psi_{\pi_GB(h)}(k)\bigr)$$
and obtain
$$\pi_GB(h)\pi_GB(k) = \pi_GB\bigl(h\Psi_{\pi_GB(h)}(k)\bigr),$$
which shows that $\pi_GB|_H$ is indeed a relative Rota--Baxter operator.
\end{proof}

We will now provide an example that shows  that a Rota--Baxter operator on $H\rtimes_{\Psi} G$ does not necessarily commute with the projection $\pi_G\colon H \rtimes_{\Psi} G \rightarrow G,$ and thus, does not necessarily induce a Rota--Baxter operator on $G.$ 

%{\color {blue} It is no a homomorphism of RRB-groups. Do you know a definition of a homomorphism of RRB-groups?}

\begin{example}
For the group $S_3=A_3\rtimes \gen{s_1}$ consider the Rota--Baxter operator $B\colon S_3\rightarrow S_3,$ defined as
$$\xymatrix @R=1pc @C=1pc {
B(1) = 1, & B(s_1) = s_1s_2, & B(s_2) = 1, \\
B(s_1s_2) = s_2s_1,  & B(s_2s_1) = s_1 s_2, & B(s_1s_2s_1) = s_2s_1.
}$$

Note that $\pi_{\gen{s_1}}\bigl(B(s_1)\bigr) = \pi_{\gen{s_1}}(s_1s_2) = 1$ and $B\bigl(\pi_{\gen{s_1}}(s_1)\bigr)=B(s_1)=s_1s_2,$ which means that $\pi_{\gen{s_1}} B\neq B\pi_{\gen{s_1}}.$
\end{example}

We will now provide an example, that shows, the the construction of Theorem \ref{R-RR} can map different Rota--Baxter operators to the same relative Rota--Baxter operator.

\begin{example}
Let $G$ and $H$ be groups, and $H$ be abelian. Let $\Psi$ be an action of $G$ on $H.$ Define Rota--Baxter operators $B_{-1}$ and $B_e$ on $H\rtimes G$ by the following way: $B_{-1}(x)=x^{-1},$ $B_e(x)=e.$ In general, the Rota--Baxter groups, defined by these operators, are not isomorphic, which can be checked by applying Proposition \ref{prop RBO to skew braces}. Note that the image of $\pi_H B_{-1}|_H$ lies in the center of $H,$ because $H=Z(H)$, and the image of $\pi_H B_e|_H$ is trivial. At the same time,
$$\pi_GB_{-1}(h)=\pi_G(h^{-1})=e=\pi_GB_e(h),$$
which means that the relative Rota--Baxter operators, obtained by applying Theorem \ref{R-RR} to operators $B_{-1}$ and $B_e,$ are equal.
\end{example}

\medskip

By analogy with RB-operators of weight $-1$ (see Definition \ref{Def:1}) we introduce RRB-operators of weight $-1$.

\begin{definition}
Let $G$ and $H$ be groups and $\Psi\colon G\rightarrow \Aut H$ be an action of $G$ on $H.$ A~map $C \colon H\rightarrow G$ is called a \textit{relative Rota--Baxter operator} of weight $-1$ with respect to  $(G, \Psi)$, if 
$$
C(h) \cdot C(k)=C\bigl(\Psi_{C(h)}(k) \cdot h\bigr),~~~h, k \in H.$$
\end{definition}

\begin{proposition} \label{PM}
Let $G$ and $H$ be groups, $\Psi\colon G\rightarrow \Aut H$ be an action, and $B\colon H\rightarrow G$ be a relative Rota--Baxter operator. Then

1)  the map $C,$ defined as $C(h)=B(h^{-1}),$ $h \in H$ is a relative Rota--Baxter operator of weight $-1.$

2) If $\varphi\in \Aut H,$ $\psi\in \Aut G,$ and for any $g\in G$ the following equality holds: 
$$\varphi^{-1}\Psi_g\varphi=\Psi_{\psi(g)},$$
then the composition $\psi B \varphi$ is a relative Rota--Baxter operator.
\end{proposition}

\begin{proof}
1) We can check directly: for any $h,k\in H$ we have
$$C(h)C(k)=B(h^{-1})B(k^{-1})=B\bigl(h^{-1}\Psi_{B(h^{-1})}(k^{-1})\bigr)=\bigl((\Psi_{C(h)}(k)h)^{-1}\bigr) = C\bigl(\Psi_C(h)(k)h\bigr).$$

2) For any $h,k\in H$ we have
$$\psi B \phi(h)\cdot \psi B \phi(k) = \psi B\bigl(\phi(h)\Psi_{B\phi(h)}(\phi(k))\bigr)
= \psi B\phi\bigl(h\phi^{-1}\Psi_{B\phi(h)}(\phi(k))\bigr)= \psi B\phi\bigl(h\Psi_{\psi B\phi(h)}(k)\bigr).$$
\end{proof}

\begin{corollary} \label{new}
Let $(H, G, \Psi, B)$ be a relative Rota--Baxter group and $\varphi\in Z\bigl(\Aut H\bigr).$
Then $(H, G, \Psi, B \varphi)$  is a relative Rota--Baxter group.
\end{corollary}
\begin{proof}
For any $g\in G$ the automorphism $\Psi_g\in \Aut H$ commutes with $\varphi,$ and we have $\phi^{-1}\Psi \phi=\Psi_g.$ Therefore, $\psi B \phi$ is a relative Rota--Baxter operator, where $\psi$ is the identity automorphism. 
\end{proof}

\bigskip

%%%%%%%%%%%%%%%%%%%%%%%%%%%%%%%%%%%%%%%%%%%%%%%%%%%%%%%%%

\section{Skew braces from RB- and RRB-operators} \label{skew braces}
As we know (see Section \ref{Prel}), if $(G, \cdot)$ is a group, $B \colon G \to G$ is a RB-operator, then $(G, \cdot, \circ_B)$ is a skew left brace, which is called a {\it Rota-Baxter skew left brace}, where
$$
a \circ_B b = a B(a) b B(a)^{-1},~~~a, b \in G.
$$
 The following lemma is evident.

\begin{lemma} \label{AbSB}
If $(G, \cdot)$ is an abelian group, then

1) any RB-operator on $G$ is an endomorphism,

2) Any Rota--Baxter skew left brace $(G, \cdot, \circ_B)$ is trivial which means $a \circ_B b = a \cdot b$ for any $a, b \in G$.
\end{lemma}

\begin{theorem}[\textbf{\cite{China_2023}}]\label{th RRBO to post-group}
Let $B\colon H\rightarrow G$ be a relative Rota--Baxter operator with respect to  $(G, \Psi)$. Put
$$h\triangleright k = \Psi_{B(h)}(k).$$
for any $h,k\in H.$
Then $(H,\cdot,\triangleright)$ is a post-group.
\end{theorem}

From  Theorem \ref{th RRBO to post-group} and Theorem \ref{th braces <-> post-groups} (see, also \cite[Proposition 3.5]{RS-23}) it follows that if we define a new operation $\circ_B \colon H \to H$,
$$
h\circ_B k = h\Psi_{B(h)}(k),~~~h, k \in H,
$$
using a relative Rota--Baxter operator $B\colon H\rightarrow G$ with respect to $(G, \Psi)$, then $(H, \cdot, \circ_B)$
is a skew left brace.
 The following example compares construction of RB skew left braces and RRB skew left braces.

\begin{example} \label{SLB}
Let $H = \Z_4$ be a cyclic group of order $4$. Then, by Lemma \ref{AbSB}(1) there are following RB-operators on  $\Z_4$:

1) $B_0(h) =  0$ for any $h \in H$;

2) $B_{-1}$, which acts by the rules $0 \mapsto 0$,  $1 \mapsto 3$, $2 \mapsto 2$, $3 \mapsto 1$;

3) $B_{2}$, which acts by the rules $0 \mapsto 0$,  $1 \mapsto 2$, $2 \mapsto 0$, $3 \mapsto 2$.

By Lemma \ref{AbSB}(2), on $\Z_4$ there exists only trivial RB skew left brace.

Now, let us construct relative Rota--Baxter skew left braces on $H = \Z_4$ with respect to $(G = \Z_2\times \Z_2, \Psi)$.
Note that $\Aut \Z_4 = \set{\epsilon, -\epsilon},$ where $\epsilon=\id$ and $-\epsilon(g)=-g$ for any $g\in \Z_4.$
Let $\Psi \colon \Z_2\times \Z_2 \to \Z_4$ be defined by the following way:
$$
\Psi(0,0) =\epsilon;\; \Psi(1,0)=-\epsilon;\; \Psi(0,1)=\epsilon;\; \Psi(1,1)=-\epsilon.
$$

Define $B\colon \Z_4\rightarrow \Z_2\times \Z_2$ as follows:
$$
B(0)=(0,0),\; B(1) = (1,0),\; B(2) = (0,1),\; B(3) = (1,1).
$$
One can check that $B$ is a relative Rota--Baxter operator with respect to $(G, \Psi)$ and 
by applying Theorems \ref{th RRBO to post-group} and \ref{th braces <-> post-groups} to the operator $B,$ we get a skew left brace $(\Z_4,+,\circ_B),$ where $(\Z_4,+)\cong \Z_4$ and $(\Z_4,\circ_B)\cong \Z_2\times \Z_2.$

 Hence, using RRB-operators we can construct more skew braces, than using only RB-operators.
\end{example}

It is interesting to generalize this example, by taking  $H = \Z_{p^2}$, $G = \Z_p\times \Z_p$, where $p$ is a prime number.
The following theorem shows that for $p>2$ the set of skew left braces which can be defined on $H$ using RRB-operators of the form $B\colon \Z_{p^2}\rightarrow \Z_p\times \Z_p$ is the same as using RB-operators  $B' \colon \Z_{p^2} \to \Z_{p^2}$.

\begin{theorem} \label{Zp}
Let $p$ be a prime number and $\Psi \colon \Z_p\times \Z_p \to \Aut \Z_{p^2}$ be a group homomorphism. For any $b_1,b_p\in \Z_p\times \Z_p,$ there is no more than one relative Rota-Baxter operator $B\colon \Z_{p^2}\rightarrow \Z_p\times \Z_p$ such that $B(1)=b_1$ and $B(p)=b_p.$

Moreover, if $p>2,$ then any relative Rota-Baxter operator $B\colon \Z_{p^2}\rightarrow \Z_p\times \Z_p$ is a homomorphism.
\end{theorem}

\begin{proof}
Since $\Aut \Z_{p^2}$ is a group of order $p(p-1),$ then $\Aut \Z_{p^2} \simeq Z_p \times A$, where $A$ is an abelian group of order $p - 1$. Let us consider an automorphism $\chi_k$ of a group $\Z_{p^2}$ defined as 
\[\chi_k \colon 1 \mapsto kp + 1\]
then $\chi_k$ is an element of order $p$ i.e.
\[(kp + 1)^p = (kp)^p + C^1_p(kp)^{p-1} + \ldots + C^{p-2}_p(kp)^2 + C^{p-1}_pkp + 1,~\mbox{where}~C_m^l=\frac{m!}{l!(m-l)!}. \]
All elements of the sum above obviously divided by $p^2$ except $1$.
Thus $\chi_k$  is an element of order $p$. And there is only $p$ such elements (we can take $k = 0, 1, \ldots, p-1$).

 It follows that any action $\Psi\colon \Z_p \times \Z_p \rightarrow \Aut\Z_{p^2}$ has the form
$$\Psi_{(n_1,n_2)}x=\bigl(p(k_1n_1+k_2n_2)+1\bigr)x,$$
where the numbers $k_1,k_2\in\{0,1,\ldots,p-1\}$ define the action.

Now fix the action $\Psi$ and suppose that $B\colon \Z_{p^2}\rightarrow \Z_p\times \Z_p$ is a relative Rota--Baxter operator,
where  $B(x) = (x_1, x_2)$ for some $x_1, x_2 \in \Z_p$. Let  $\cdot$ denote the scalar multiplication of vectors  from $\Z_p^2$, i.e.  $(a_1, a_2)\cdot(b_1, b_2) = a_1b_1 + a_2b_2$. Also, define a function $t\colon \Z_{p^2}\rightarrow p\Z_{p^2}$ as 
$$
t(x)=pB(x) \cdot (k_1, k_2) = p(k_1x_1+k_2x_2).
$$
Note that $\Psi_{B(x)}y=\bigl(t(x)+1\bigr)y$ and since $B$ is a relative Rota--Baxter operator, we have 
$$
B(x)+B(y)=B\bigl(x+\Psi_B(x)y\bigr)=B\bigl(x+y+t(x)y).
$$
Using the fact  that  $t(x)$ is divisible by $p,$ we can write 
$$B(x)+B(py)=B(x+py+pt(x)y)=B(x+py).$$ It follows that the restriction $B|_{p\Z_{p^2}}$ is a homomorphism.

We will now prove by induction over $n$ that 
$$nB(x)=B\biggl(\frac{\bigl(t(x)+1\bigr)^n-1}{t(x)}x\biggr).$$
Indeed, for $n=1$ we have $B(x)=B(x),$ and if the statement holds for $n-1,$ then 
\begin{gather*}
B(x)+(n-1)B\bigl(x\bigr)=B\biggl(x+(t(x)+1)\frac{\bigl(t(x)+1\bigr)^{n-1}-1}{t(x)}x\biggr) = \\
= B\biggl(x+\frac{\bigl(t(x)+1\bigr)^n-t(x)-1}{t(x)}x\biggr) = B\biggl(x+\frac{\bigl(t(x)+1\bigr)^n-t(x)-1}{t(x)}x\biggr) = \\ = B\biggl(\frac{\bigl(t(x)+1\bigr)^n-1}{t(x)}x\biggr).
\end{gather*}

Now note that 
$$
\frac{\bigl(1+t(x)\bigr)^n-1}{t(x)}=n+\sum\limits_{i=2}^n C_n^it(x)^{i-1}.
$$
%\color{red}
%Since $t(x)\in p\Z_{p^2},$ we can define a map $s\colon \Z_{p^2}\times\mathbb{N} \rightarrow \mathbb{N}$ such that
%$$
%s(x,n)=\frac{1}{p}\sigma\sum\limits_{i=2}^n C_n^i t(x)^{i-1},
%$$
%where $\sigma$ is the inclusion $\Z_{p^2}\rightarrow \{0,1,\ldots,p^2-1\}.$
%\color{black}
 We now have 
$$
nB(x)=B(nx+ps(x,n)x)=B(nx)+B(ps(x,n)x).
$$
 Note that $B|_{p\Z_{p^2}}$ is a homomorphism of abelian groups, and we can by extension treat it as a homomorphism of $\Z_{p^2}$-modules. We can thus write 
$$
B(ps(x,n)x)=s(x,n)xB(p)~\mbox{and}~B(nx)=nB(x)-s(x,n)xB(p).
$$
By substituting $x=1,$ we obtain
\begin{equation}\label{eq RRBO on Zp^2}
B(n)=nB(1)-s(1,n)B(p).
\end{equation}

Note that $s(1,n)$ can be calculated knowing only $B(1)$ and the action $\Psi,$ so we have proven that for any given action, the values of $B(1)$ and $B(p)$ define a unique relative Rota--Baxter operator.

Now let $p>2.$ Note that $C_p^2=\frac{p!}{2!(p-2)!}$ is divisible by $p.$ Since $t(x)$ is also divisible by $p,$ it follows that $\sum\limits_{i=2}^p C_p^it(x)^{i-1}$ is divisible by $p^2.$ On one hand, $pB(x)=0,$ and on the other hand,
$$pB(x)=B\Bigl(px+\sum\limits_{i=2}^p C_p^it(x)^{i-1}\Bigr)=B(px).$$
It follows that $B(px)=0$ for any $x,$ so $B|_{p\Z_{p^2}}$ is a zero homomorphism, and in turn,
$$nB(1)=B(n)-s(1,n)B(p)=B(n),$$
which means that $B$ is a homomorphism.
\end{proof}

\begin{question}
 Let us define a Rota--Baxter operator (RB-operator) on a skew left brace as a map which is a Rota--Baxter operator on both groups of skew left brace. Find Rota--Baxter operators on skew left braces. If we are considering RB-operators on a left brace $(G, +, \circ)$, then on the group $(G, +)$ it is an endomorphism.
 \end{question}

\bigskip

%%%%%%%%%%%%%%%%%%%%%%%%%%%%%%%%%%%%%%%%%%%%%

\section{Skew left braces, nilpotent groups and the YBE} \label{YBE}

\subsection{$\lambda$-homomorphic skew left braces}

Consider a particular type of skew left braces, which was introduced in \cite{lambda-homomorphic braces}.
A skew left brace $(G,\cdot,\circ)$ is called $\lambda$-\textit{homomorphic}, if $\lambda\colon (G,\cdot) \rightarrow \Aut(G,\cdot)$ is a group homomorphism.
The main idea for the introduction of $\lambda$-homomorphic skew left braces is the following. If we take a group $G$ with a generating set $A$, and define a map 
$\lambda \colon A \to \Aut(G)$, then we can extend it on all elements of $G$. Under some conditions this map $\lambda  \colon G \to \Aut(G)$ is a  $\lambda$-map of a skew left brace $(G, \cdot, \circ)$, where  the second operation is defined by the rule
$$
a \circ b = a \cdot \lambda_a(b),~~a, b \in G. 
$$

Class of $\lambda$-homomorphic skew left braces is not a big class, but it  has a good description. More precisely, any  $\lambda$-homomorphic skew left brace is metatrivial that means that it is an extension of one trivial skew left brace by another trivial skew left brace (see \cite{lambda-homomorphic braces}).

We introduce the following definition.

\begin{definition}
A post-group  $(G,\cdot,\triangleright)$ is said to be a \textit{homomorphic post-group} if it satisfies the identity
 $$(a\cdot b)\triangleright c = (a\triangleright c)\cdot(b\triangleright c)$$
for all $a, b, c \in G$. 
\end{definition}

Note that the condition $(a\cdot b)\triangleright c = (a\triangleright c)\cdot(b\triangleright c)$ is the  right distributivity.

By applying Theorem \ref{th braces <-> post-groups} to $\lambda$-homomorphic skew left braces we  obtain the following result.

\begin{proposition} \label{LSB-PG}
Let $(G,\cdot,\circ)$ be a $\lambda$-homomorphic skew left brace. Then the post-group $(G,\cdot, \triangleright)$ has the following properties:

1) $(G,\cdot,\triangleright)$ is a homomorphic post-group.

2) $[a,b,c]=a\triangleright c,$ where $[a,b,c]$ is the associator: 
$$[a,b,c]:=\bigl(a\triangleright(b\triangleright c)\bigr)\cdot\bigl((a \triangleright b) \triangleright c\bigr)^{-1}.$$
\end{proposition}
\begin{proof}
1) Follows from $$(a\cdot b)\triangleright c = \lambda_{a\cdot b}(c)=\lambda_a(c)\cdot \lambda_b(c) = (a\triangleright c)\cdot(b\triangleright c).$$

2) We have
$$
(a\triangleright c)\cdot \bigl((a \triangleright b) \triangleright c\bigr) = 
\bigl(a\cdot (a\triangleright b)\bigr)\triangleright c = a\triangleright (b\triangleright c),
$$
where the first equality follows from the definition of a homomorphic post-group and the second one follows from the definition of post-group (see Definition \ref{PG}).

\end{proof}

We will now use Theorem \ref{th braces <-> post-groups} to construct a particular class of skew left braces on two-step nilpotent groups.

\begin{proposition}\label{prop 2-step nilpotent post-group}
1) For a group $(G,\cdot)$ let $a\triangleright b = a^{-1}ba.$ Then $(G,\cdot,\triangleright)$ is a post-group.

2) For a two-step nilpotent group $(G,\cdot)$ and $n\in \Z\setminus\{0\}$ let $a\triangleright b = a^{-n}ba^n.$ Then $(G,\cdot,\triangleright)$ is a post-group.
\end{proposition}
\begin{proof}
Since conjugation by an element is always an automorphism of the group, we only have to show that $a\triangleright (b\triangleright c)=\bigl(a(a\triangleright b)\bigr)\triangleright c.$

1) If $a\triangleright b = a^{-1} b a,$ then 
$$a\triangleright (b\triangleright c) = a^{-1}b^{-1}cba,$$
and
$$\bigl(a(a\triangleright b)\bigr)\triangleright c = (ba)\triangleright c = a^{-1}b^{-1}cba.$$

2) If $G$ is a two-step nilpotent group and $a\triangleright b = a^{-n} b a^n = b[b,a]^n,$ then we have
$$a\triangleright (b\triangleright c) = a\triangleright \bigr(c[c,b]^n\bigl)
= c[c,b]^n\bigl[c[c,b]^n,a\bigr]=c[c,b]^n[c,a]^n\bigl[[c,b],a\bigr]^n=c[c,b]^n[c,a]^n,$$
and
$$\bigl(a(a\triangleright b)\bigr)\triangleright c = \bigl(ab[b,a]^n\bigr)\triangleright c = c\bigl[c,ab[b,a]^n\bigr]^n
= c[c,b]^n[c,a]^n\bigl[c,[b,a]^n\bigr]^n=c[c,b]^n[c,a]^n.$$
\end{proof}

Let $G$ be a two-step nilpotent group and $a\triangleright b=a^{-n}ba^n$ for some integer $n.$ By Proposition \ref{prop 2-step nilpotent post-group}, $(G,\cdot,\triangleright)$ is a post-group. By Theorem \ref{th braces <-> post-groups}, $(G,\cdot,\circ)$ is a skew left brace, where
$$a\circ b = a\cdot (a\triangleright b)=aa^{-n}ba^n=ab[b,a]^n.$$

The following statement holds for the group $(G,\circ)$

\begin{proposition}The group $(G,\circ)$ defined above is two-step nilpotent.
\end{proposition}
\begin{proof}
Note that the inverse element with respect to the operation $\circ$ is the same element as the inverse with respect to the operation $\cdot.$ Indeed, $a^{-1}\circ a = a^{-1}a[a,a^{-1}]^n=e.$
Denote by $[a,b]_\circ$ the commutator with respect to the operation $\circ$:
\begin{gather*}
[a,b]_\circ = a^{-1}\circ b^{-1}\circ a\circ b
= \bigl(a^{-1}b^{-1}[b^{-1},a^{-1}]^n\bigr)\circ\bigl(ab[b,a]^n\bigr)
= \bigl(a^{-1}b^{-1}[b,a]^n\bigr)\circ\bigl(ab[b,a]^n\bigr) = \\
= a^{-1}b^{-1}[b,a]^nab[b,a]^n\bigl[ab[b,a]^n,a^{-1}b^{-1}[b,a]^n\bigr]^n
= [a,b][b,a]^{2n}\bigl[ba[a,b],(ba)^{-1}\bigr]=[b,a]^{2n-1}.
\end{gather*}
Now we can see that $\bigl[[a,b]_\circ,c\bigr]_\circ=\bigl[c,[b,a]^{2n-1}\bigr]^{2n-1}=\bigl[c,[b,a]\bigr]^{(2n-1)^2}=e$ for any $a,b,c\in G,$ which means that the group $(G,\circ)$ is two-step nilpotent. 
\end{proof}

It is easy to show that $(G,\circ)$ is not necessarily isomorphic to $(G,\cdot).$ Indeed, if $(G,\cdot)$ satisfies the relation $[a,b]^{2n-1}=e$ for any $a$ and $b,$ then $[a,b]_\circ=[a,b]^{1-2n}=e,$ hence the group $(G,\circ)$ has to be abelian. With $n$ not equal to $0$ or $1,$ groups that satisfy the relation $[a,b]^{2n-1}=e$ do not have to be abelian.

\begin{proposition} \label{SLB-LSLB}
A skew left brace constructed above is a $\lambda$-homomorphic skew left brace.
\end{proposition}

\begin{proof}
We have to prove that the  $\lambda$-map which corresponds to skew left brace $(G, \cdot, \circ)$ is a homomorphism $\lambda \colon (G, \cdot) \to \Aut(G, \cdot)$.
By the formula after Proposition \ref{prop 2-step nilpotent post-group}, the new product is
$$
a\circ b =aa^{-n}ba^n. 
$$
Hence, $\lambda_a(b) = a^{-1} \cdot (a\circ b) =a^{-n}ba^n$ and we have
$$
\lambda_a (\lambda_b (c)) = \lambda_a (b^{-n} c b^n) = a^{-n}b^{-n} c b^{n} a^n.
$$
On the other side,
$$
\lambda_{ab} (c) = (ab)^{-n}  c (ab)^n = a^{-n} b^{-n} [a, b]^{-n(n-1)/2} c b^{n} a^{n} [a, b]^{n(n-1)/2} = a^{-n} b^{-n}  c b^{n} a^{n}. 
$$
Comparing with the previous formula, we see that $\lambda_a \lambda_b = \lambda_{ab}$ for any $a, b \in G$. It means that $\lambda$ is a homomorphism.
\end{proof}

\subsection{Verbal solutions of the Yang-Baxter equation}

Let $X$ be a nonempty set and $S\colon X^2\rightarrow X^2.$ 
The map $S$ is called a \textit{solution of the Yang--Baxter equation} on $X,$ if 
$$ S_1\,S_2\,S_1=S_2\,S_1\,S_2, $$
where
$S_1=S\times \Id,$ $S_2=\Id \times S.$

The following theorem allows us to use skew left braces in order to obtain solutions of the Yang--Baxter equation.

\begin{theorem}[\textbf{\cite{Vendramin}}]\label{Skew brace to braid equation}
Let $(G,\cdot,\circ)$ be a skew left brace. Then the map $S\colon G^2\rightarrow G^2,$ defined as
$$S(a,b)=\Bigl( \lambda_a(b), \overline{\lambda_a(b)}\circ a\circ b \Bigr),$$
where $\overline{x}$ is the inverse of $x$ with respect to the operation $\circ,$ is a solution of the Yang-Baxter equation on the set $G.$
\end{theorem}

In the  previous section we constructed some skew left braces on nilpotent groups. 
We will now proceed to use Theorem \ref{th braces <-> post-groups}, and \ref{Skew brace to braid equation} to construct  solutions to the Yang--Baxter equation on two-step nilpotent groups.
We will now explore verbal solutions of the Yang--Baxter equation on two-step nilpotent groups.

\begin{definition}
For a group $G,$ a map $\phi\colon G^n \rightarrow G$ is called a \emph{verbal map} if there is a group word $w = w(x_1,\ldots,x_n)$ on $n$ letters such that for any $g_1,\ldots,g_n\in G$ we have $\phi(g_1,\ldots,g_n) = w(g_1,\ldots,g_n)$.

For any group word $w$ we will denote the verbal map obtained in this way by $\phi_w.$
\end{definition}

\begin{definition}
Let $G$ be a group. A solution $S$ of the Yang--Baxter equation on $G$ is called a \textit{verbal solution} if there are group words $w_1$ and $w_2$ such that $S=\phi_{w_1}\times \phi_{w_2}.$
\end{definition}

Note that in a two-step nilpotent group $G$ any verbal map $\phi$ has a nice standard form $\phi(x,y)=x^ay^b[y,x]^m,$ and that even though it needs not be a group homomorphism, it has a well-defined abelianization $\phi^{Ab}(x,y)=x^ay^b,$ and the following diagram commutes:
$$
\xymatrix{
G^2 \ar^\phi[r] \ar[d] & G \ar[d] \\
(G^{Ab})^2 \ar^{\phi^{Ab}}[r] & G^{Ab}
}
$$

The following proposition is immediate from this:
\begin{proposition}
If $S=\phi_{w_1}\times \phi_{w_2}$ is a verbal solution of the Yang--Baxter equation on a two-step nilpotent group $G,$ then $S^{Ab}=\phi_{w_1}^{Ab}\times \phi_{w_2}^{Ab}$ is a verbal solution of the Yang--Baxter equation on $G^{Ab}.$
\end{proposition}

Verbal maps $(G^{Ab})^2\rightarrow (G^{Ab})^2$ can be represented as matrices with integer coefficients, and for $2\times 2$ matrices we can fully describe which of them satisfy the Yang--Baxter equation:

\begin{theorem}\label{th YBE over integral domain}
Let $M$ be a $2\times 2$ matrix with coefficients in an integral domain $R$. The map from $R^2$ to $R^2$ defined by left multiplication by $M$ is a solution of the Yang-Baxter equation on $R$ if and only if $M$ has at least one of the following forms:

$$
\begin{pmatrix}
1-bc & b \\
c & 0
\end{pmatrix}
;\:
\begin{pmatrix}
0 & b \\
c & 1-bc
\end{pmatrix}
;\:
\begin{pmatrix}
0 & b \\
c & 0
\end{pmatrix}
;\:
\begin{pmatrix}
1 & 0 \\
0 & 1
\end{pmatrix}.
$$
\end{theorem}
\begin{proof}

We can write down the Yang--Baxter equation in the following form:
\begin{gather*}
0=
\begin{pmatrix}
a & b & 0 \\
c & d & 0 \\
0 & 0 & 1
\end{pmatrix}
\begin{pmatrix}
1 & 0 & 0 \\
0 & a & b \\
0 & c & d
\end{pmatrix}
\begin{pmatrix}
a & b & 0 \\
c & d & 0 \\
0 & 0 & 1
\end{pmatrix}
-
\begin{pmatrix}
1 & 0 & 0 \\
0 & a & b \\
0 & c & d
\end{pmatrix}
\begin{pmatrix}
a & b & 0 \\
c & d & 0 \\
0 & 0 & 1
\end{pmatrix}
\begin{pmatrix}
1 & 0 & 0 \\
0 & a & b \\
0 & c & d
\end{pmatrix}
= \\ =
\begin{pmatrix}
a(a+bc-1) & abd & 0 \\
acd & ad(d-a) & -abd \\
0 & -acd & -d(d+bc-1)
\end{pmatrix}
\end{gather*}
and obtain the following system of algebraic equations:
\begin{align*}
&abd=0; \\
&acd=0; \\
&a(a+bc-1)=0; \\
&d(d+bc-1)=0; \\
&ad(d-a)=0.
\end{align*}
Since the coefficients are taken from a ring with no zero divisors, the solution of the system can be decomposed into a union of solutions of four simpler systems of equations:
\begin{align*}
&1)\; a=0,\: d=1-bc; \\
&2)\; d=0,\: a=1-bc; \\
&3)\; a=0,\: d=0; \\
&4)\; b=0,\: c=0,\: a(a-1)=0,\: d(d-1)=0.
\end{align*}
Note that the solution of the system $4)$ is a union of $4$ points, $3$ of which are also solutions of $1),$ $2)$ or $3).$ With this in mind, $4)$ can be reduced to $a=1,b=1,c=0,d=0,$ which completes the proof.
\end{proof}

We will now investigate verbal solutions of the Yang--Baxter equation on two-step nilpotent groups. We are interested in such pairs of group words $w_1(x,y)=x^ay^b[y,x]^m,$ $w_2(x,y)=x^cy^d[y,x]^n$ that the map $S=\phi_{w_1}\times \phi_{w_2}$ is a solution of the Yang--Baxter equation on any two-step nilpotent group $G.$ If $w_1$ and $w_2$ are such words, then the maps
$(S\times \Id)(\Id\times S)(S\times \Id)$ and $(\Id\times S)(S\times \Id)(\Id\times S)$ from $F^3$ to $F^3$ must coincide for any free two-step nilpotent group $F$.

Abelianization of a free two-step nilpotent group is a free abelian group, so $S^{Ab}$ must be a solution of the Yang--Baxter equation on $\Z,$ and as such, the matrix $M_S=\begin{pmatrix}a& b \\ c& d\end{pmatrix}$ must be of at least one of the forms listed in theorem \ref{th YBE over integral domain}.
We will denote $S_1=S\times \Id,$ $S_2=\Id\times S$ and write down the corresponding Yang--Baxter equation for each of these matrices with free parameters $m$ and $n$.

Starting with $M_S=\begin{pmatrix}1& 0 \\ 0& 1\end{pmatrix},$ we have 

\begin{align*}
&S_1S_2S_1
\begin{pmatrix}
x \\
y \\
z
\end{pmatrix}
= S_1S_2
\begin{pmatrix}
x[y,x]^m \\
y[y,x]^n \\
z
\end{pmatrix}
= S_1
\begin{pmatrix}
x[y,x]^m \\
y[y,x]^n[z,y]^m \\
z[z,y]^n
\end{pmatrix}
=
\begin{pmatrix}
x[y,x]^{2m} \\
y[y,x]^{2n}[z,y]^m \\
z[z,y]^n
\end{pmatrix}; \\
&S_2S_1S_2
\begin{pmatrix}
x \\
y \\
z
\end{pmatrix}
=S_2S_1
\begin{pmatrix}
x \\
y[z,y]^m \\
z[z,y]^n
\end{pmatrix}
=S_2
\begin{pmatrix}
x[y,x]^m \\
y[y,x]^n[z,y]^m \\
z[z,y]^n
\end{pmatrix}
=
\begin{pmatrix}
x[y,x]^m \\
y[y,x]^{n}[z,y]^{2m} \\
z[z,y]^{2n}
\end{pmatrix}.
\end{align*}

The Yang--Baxter equation here implies $n=0$ and $m=0,$ so the only verbal solution corresponding to this matrix is

$$S(x,y)=(x,y).$$

Now assume $M_S=\begin{pmatrix}0& b \\ c& 0\end{pmatrix}.$ We have

\begin{align*}
&S_1S_2S_1
\begin{pmatrix}
x \\
y \\
z
\end{pmatrix}
=S_1S_2
\begin{pmatrix}
y^b[y,x]^m \\
x^c[y,x]^n \\
z
\end{pmatrix}
=S_1
\begin{pmatrix}
y^b[y,x]^m \\
z^b[z,x]^{cm} \\
x^{c^2}[y,x]^{cn}[z,x]^{cn}
\end{pmatrix}
=
\begin{pmatrix}
z^{b^2}[z,x]^{bcm}[z,y]^{b^2m} \\
y^{bc}[y,x]^{cm}[z,y]^{b^2n} \\
x^{c^2}[y,x]^{cn}[z,x]^{cn}
\end{pmatrix}; \\
&S_2S_1S_2
\begin{pmatrix}
x \\
y \\
z
\end{pmatrix}
=S_2S_1
\begin{pmatrix}
x \\
z^b[z,y]^m \\
y^c[z,y]^n
\end{pmatrix}
=S_1
\begin{pmatrix}
z^{b^2}[z,x]^{bm}[z,y]^{bm} \\
x^c[z,x]^{bn} \\
y^c[z,y]^n
\end{pmatrix}
=
\begin{pmatrix}
z^{b^2}[z,x]^{bm}[z,y]^{bm} \\
y^{bc}[y,x]^{c^2m}[z,y]^{bn} \\
x^{c^2}[y,x]^{c^2n}[z,x]^{bcn}
\end{pmatrix}.
\end{align*}
The Yang-Baxter equation in this case is equivalent to the following system of algebraic equations:
\begin{align*}
b(c-1)m=0; \\
b(b-1)m=0; \\
c(c-1)m=0; \\
b(b-1)n=0; \\
c(c-1)n=0; \\
c(b-1)n=0.
\end{align*}
For the sake of uniformity, we will rename the free parameters to $u$ and $v.$ With that in mind, the set of solutions to the system of algebraic equations above and the corresponding verbal solutions $S$ is as follows:
\begin{align*}
b=0,\: c=0:\;\; &S(x,y)=\bigl([y,x]^u,[y,x]^v\bigr);\\
b=0,\: c=1,\: n=0:\;\; &S(x,y)=\bigl([y,x]^u,x\bigr); \\
b=1,\: c=1:\;\; &S(x,y)=\bigl(y[y,x]^u,x[y,x]^v\bigr); \\
b=1,\: c=0,\: m=0:\;\; &S(x,y)=\bigl(y,[y,x]^u\bigr); \\
m=0,\: n=0:\;\; &S(x,y)=\bigl(y^u,x^v\bigr).
\end{align*}

Now assume $M_S=\begin{pmatrix}1-bc& b \\ c& 0\end{pmatrix}.$ Note that in a two-step nilpotent group the following expression holds: $$(xy)^k=x^ky^k[y,x]^{\frac{1}{2}k(k-1)},$$ which can be proven by induction. Indeed, for $k=0$ the expression holds. If the expression holds for $k,$ then 
$$(xy)^{k+1}=x^ky^kxy[y,x]^{\frac{1}{2}k(k-1)}=x^{k+1}y^{k+1}[y,x]^{\frac{1}{2}k(k-1)+k}=x^{k+1}y^{k+1}[y,x]^{\frac{1}{2}(k+1)(k+1-1)}$$
and
$$(xy)^{k-1}=x^ky^ky^{-1}x^{-1}[y,x]^{\frac{1}{2}k(k-1)}=x^{k-1}y^{k-1}[y,x]^{\frac{1}{2}k(k-1)-(k-1)}=x^{k-1}y^{k-1}[y,x]^{\frac{1}{2}(k-1)(k-2)}.$$ 
Now, for the map $S$ we have

\begin{gather*}
S_1S_2S_1
\begin{pmatrix}
x \\
y \\
z
\end{pmatrix}
=
S_1S_2
\begin{pmatrix}
x^{1-bc}y^b[y,x]^m \\
x^c[y,x]^n \\
z
\end{pmatrix}
=S_1
\begin{pmatrix}
x^{1-bc}y^b[y,x]^m \\
(x^c[y,x]^n)^{1-bc}z^b[z,x^c]^m \\
x^{c^2}[y,x]^{cn}[z,x^c]^n
\end{pmatrix}
= \\ =S_1
\begin{pmatrix}
x^{1-bc}y^b[y,x]^m \\
x^{c(1-bc)}z^b[y,x]^{(1-bc)n}[z,x]^{cm} \\
x^{c^2}[y,x]^{cn}[z,x]^{cn}
\end{pmatrix}
= \\ =
\begin{pmatrix}
(x^{1-bc}y^b[y,x]^m)^{1-bc}(x^{c(1-bc)}z^b[y,x]^{(1-bc)n}[z,x]^{cm})^b[x^{c(1-bc)}z^b,x^{1-bc}y^b]^m\\
(x^{1-bc}y^b[y,x]^m)^{c}[x^{c(1-bc)}z^b,x^{1-bc}y^b]^n\\
x^{c^2}[y,x]^{cn}[z,x]^{cn}
\end{pmatrix}
%= \\ =
%\begin{pmatrix}
%(x^{(1-bc)^2} y^{b(1-bc)}
%[y,x]^{(1-bc)m-\frac{1}{2}b^2c(1-bc)^2})
%(x^{bc(1-bc)}z^{b^2}
%[y,x]^{b(1-bc)n}
%[z,x]^{bcm+\frac{1}{2}b^2c(1-bc)(b-1)})
%[x^{c(1-bc)}z^b,x^{1-bc}y^b]^m \\
%x^{c(1-bc)}y^{bc}[y,x]^{mc+\frac{1}{2}bc(1-bc)(c-1)}[x^{c(1-bc)}z^b,x^{1-bc}y^b]^n \\
%x^{c^2}[y,x]^{cn}[z,x]^{cn}
%\end{pmatrix}
= \\ =
\begin{pmatrix}
x^{(1-bc)}y^{b(1-bc)}z^{b^2}
[y,x]^{(1-bc)m+\frac{1}{2}b^2c(1-bc)^2+b(1-bc)n-bc(1-bc)m}
[z,x]^{bcm+\frac{1}{2}b^2c(1-bc)(b-1)+b(1-bc)m}
[z,y]^{b^2m} \\
x^{c(1-bc)}y^{bc}
[y,x]^{mc+\frac{1}{2}bc(1-bc)(c-1)-bc(1-bc)n}
[z,x]^{b(1-bc)n}
[z,y]^{b^2n}
\\
x^{c^2}[y,x]^{cn}[z,x]^{cn}
\end{pmatrix};
\end{gather*}

\begin{gather*}
S_2S_1S_2
\begin{pmatrix}
x \\
y \\
z
\end{pmatrix}
=S_2S_1
\begin{pmatrix}
x \\
y^{1-bc}z^b[z,y]^m \\
y^c[z,y]^n
\end{pmatrix}
= S_2
\begin{pmatrix}
x^{1-bc}(y^{1-bc}z^b[z,y]^m)^b[y^{1-bc}z^b,x]^m \\
x^c[y^{1-bc}z^b,x]^n \\
y^c[z,y]^n
\end{pmatrix}
= \\ = S_2
\begin{pmatrix}
x^{(1-bc)}y^{b(1-bc)}z^{b^2}
[y,x]^{(1-bc)m}
[z,x]^{bm}
[z,y]^{bm+\frac{1}{2}b^2(1-bc)(b-1)} \\
x^c
[y,x]^{(1-bc)n}
[z,x]^{bn} \\
y^c[z,y]^n
\end{pmatrix}
= \\ =
\begin{pmatrix}
x^{(1-bc)}y^{b(1-bc)}z^{b^2}
[y,x]^{(1-bc)m}
[z,x]^{bm}
[z,y]^{bm+\frac{1}{2}b^2(1-bc)(b-1)} \\
(x^c[y,x]^{(1-bc)n}[z,x]^{bn})^{1-bc}
(y^c[z,y]^n)^b
[y^c,x^c]^m \\
(x^c[y,x]^{(1-bc)n}[z,x]^{bn})^c
[y^c,x^c]^n
\end{pmatrix}
= \\ =
\begin{pmatrix}
x^{(1-bc)}y^{b(1-bc)}z^{b^2}
[y,x]^{(1-bc)m}
[z,x]^{bm}
[z,y]^{bm+\frac{1}{2}b^2(1-bc)(b-1)} \\
x^{c(1-bc)}y^{(bc)}
[y,x]^{(1-bc)^2n+c^2m}
[z,x]^{b(1-bc)n}
[z,y]^{bn} \\
x^{c^2}
[y,x]^{c(1-bc)n+c^2n}
[z,x]^{bcn}
\end{pmatrix}.
\end{gather*}

The Yang--Baxter equation in this case is equivalent to the following system of algebraic equations:

\begin{align*}
&1)\: (1-bc)m+\frac{1}{2}b^2c(1-bc)^2+b(1-bc)n-bc(1-bc)m-(1-bc)m=0; \\
&2)\: bcm+\frac{1}{2}b^2c(1-bc)(b-1)+b(1-bc)m-bm=0; \\
&3)\: b^2m -bm -\frac{1}{2}b^2(1-bc)(b-1) =0; \\
&4)\: cm+\frac{1}{2}bc(1-bc)(c-1)-bc(1-bc)n=0; \\
&5)\: b^2n-bn=0; \\
&6)\: cn-c(1-bc)n-c^2n=0; \\
&7)\: cn-bcn=0.
\end{align*}
Or, in an alternative form,
\begin{align*}
&1)\: b(1-bc)(\frac{1}{2}bc+n-cm)=0; \\
&2)\: bc(b-1)(m+\frac{1}{2}b(1-bc))=0; \\
&3)\: b(b-1)(m-\frac{1}{2}b(1-bc))=0; \\
&4)\: \frac{1}{2}bc(1-bc)(c-1)-c(c-1)m-(1-bc)n=0; \\
&5)\: b(b-1)n=0; \\
&6)\: c^2(b-1)n=0; \\
&7)\: c(b-1)n=0.
\end{align*}

Equations $5)$ -- $7)$ all hold if and only if at least one of the following conditions is satisfied: either $b=c=0,$ or $b=1,$ or $n=0.$ We will examine these three cases separately, making corresponding substitutions into equations $1)$--$4)$.

Case $b=c=0.$ Equations $1)$--$3)$ hold automatically, and $4)$ is reduced to $n=0.$ We have $1-bc=1$ and $m$ a free parameter. This gives us the verbal solution $$S(x,y)=\bigl(x[x,y]^u,1\bigr).$$

Case $b=1.$ Equations $2)$ and $3)$ hold automatically, and we are left with the system
\begin{align*}
&1)\: (1-c)(\frac{1}{2}c+n-cm)=0;\\
&4)\: (c-1)(-\frac{1}{2}c(c-1)+n-cm)
\end{align*}
If $c=1,$ then the case is reduced to the previously examined case with the matrix $\begin{pmatrix}0 & 1 \\ 1 & 0\end{pmatrix}.$ If $n=cm-\frac{1}{2}c$ and $n=cm+\frac{1}{2}c(c-1),$ then we have $c(c-1)+c=0,$ hence $c=0$ and $n=0,$ and the case is reduced to a particular example of the case $n=0,$ which we will examine next.

Case $n=0.$ By substituting $n=0$ into equations $1)$--$4)$ we get the system
\begin{align*}
&1)\: bc(1-bc)(\frac{1}{2}b-m)=0; \\
&2)\: bc(b-1)(m+\frac{1}{2}b(1-bc))=0; \\
&3)\: b(b-1)(m-\frac{1}{2}b(1-bc))=0; \\
&4)\: -c(c-1)(m-\frac{1}{2}b(1-bc))=0. \\
\end{align*}
Equations $3)$ and $4)$ hold if and only if at least one of the following conditions holds: either $b,c\in \{0,1\},$ or $m=\frac{1}{2}b(1-bc).$

If $b=c=0$ or $b=c=1,$ the case is reduced to one of the previously examined cases. If $b=c-1=0$ or $c=b-1=0,$ then equations $1)$ and $2)$ hold automatically, $m$ stays a free parameter, and we get two new verbal solutions:
\begin{align*}
&S(x,y)=(x[y,x]^u,x);\\
&S(x,y)=(xy[y,x]^u,1).
\end{align*}

If $m=\frac{1}{2}b(1-bc),$ then the system is further reduced to
\begin{align*}
&1)\: \frac{1}{2}b^3c^2(1-bc)=0; \\
&2)\: b^2c(b-1)(1-bc)=0, \\
\end{align*}
which holds if and only if $b=0$ or $c=0$ or $1-bc=0.$ If $1-bc=0,$ then $m=0,$ and the case is reduced to a previously examined case. If $b=0,$ then $m=0,$ $c$ is a free parameter, and we have the verbal solution
$$S(x,y)=(x,x^u).$$
Finally, if $c=0,$ then $m$ is a free parameter, $b=2m,$ and we have the verbal solution
$$S(x,y)=(xy^{2u}[y,x]^u,1).$$

As for the matrix $\begin{pmatrix}0 & b \\ c & 1-bc\end{pmatrix},$ we will obtain the corresponding verbal solutions by using the symmetries of the Yang--Baxter equation.

\begin{lemma}
Let $X$ be a set and $S\colon X^2\rightarrow X^2$ is a solution of the Yang--Baxter equation on $X.$ Then $S^\sigma=\sigma S \sigma$ is a solution of the Yang--Baxter equation on $X,$ where $\sigma(x,y)=(y,x).$
\end{lemma}
\begin{proof}
Define the map $\tau\colon X^3\rightarrow X^3$ the following way: $\tau(x,y,z)=(z,y,x).$ We can assume that $S(x,y)=(f(x,y),g(x,y)).$ Then $S^\sigma(x,y)=(g(y,x),f(y,x)).$ Note that
$$\tau (S\times \Id)\tau (x,y,z)=\tau(f(z),g(y),x)=(x,g(y),f(z))=(\Id\times S^\sigma)(x,y,z).$$
Similarly, we have $\tau(\Id\times S)\tau=S^\sigma\times \Id.$ Now,
\begin{gather*}
(S^\sigma\times \Id)(\Id\times S^\sigma)(S^\sigma\times \Id) = \tau(\Id\times S)(S\times \Id)(\Id\times S)\tau; \\
(\Id\times S^\sigma)(S^\sigma\times \Id)(\Id\times S^\sigma) = \tau (S\times \Id)(\Id\times S)(S\times \Id) \tau, 
\end{gather*}
and since $S$ is a solution of the Yang--Baxter equation, the right sides of these equalities coincide, and hence
$$(S^\sigma\times \Id)(\Id\times S^\sigma)(S^\sigma\times \Id)=(\Id\times S^\sigma)(S^\sigma\times \Id)(\Id\times S^\sigma).$$
\end{proof}
\begin{corollary}
If $S(x,y)=\bigl(x^ay^b[y,x]^m,x^cy^d[y,x]^n\bigr)$ is a verbal solution of the Yang--Baxter equation on a 2-step nilpotent group, then $\bar{S}(x,y)=\bigl(x^dy^c[y,x]^{dc-n},x^by^a[y,x]^{ab-m}\bigr)$ is also a verbal solution.
\end{corollary}

Now, by combining all the solutions obtained and applying the symmetries, we can finally formulate the theorem.
\begin{theorem} \label{Sol}
If $(w_1,w_2)$ is a pair of group words on two letters such that for any two-step nilpotent group $G$ the induced map $S\colon G^2\rightarrow G^2,$ $S(x,y)=\bigl(w_1(x,y),w_2(x,y)\bigr)$ is a solution of the Yang--Baxter equation, then there are $u,v\in\Z$ such that $S(x,y)$ has one (or more) of the following forms:
\begin{align*}
&S(x,y)=(x,y);\\
&S(x,y)=\bigl([y,x]^u,[y,x]^v\bigr);\\
&S(x,y)=\bigl(y[y,x]^u,x[y,x]^v\bigr); \\
&S(x,y)=\bigl(y^u,x^v\bigr); \\
&S(x,y)=\bigl([y,x]^u,x\bigr);
&&S(x,y)=\bigl(y,[y,x]^u\bigr);\\
&S(x,y)=\bigl(x[y,x]^u,1\bigr); 
&&S(x,y)=\bigl(1,y[y,x]^u\bigr); \\
&S(x,y)=\bigl(x[y,x]^u,x\bigr);
&&S(x,y)=\bigl(y,y[y,x]^u\bigr); \\
&S(x,y)=\bigl(xy[y,x]^u,1\bigr); 
&&S(x,y)=\bigl(1,xy[y,x]^u\bigr); \\
&S(x,y)=(x,x^u);
&&S(x,y)=(y^u,y);\\
&S(x,y)=\bigl(xy^{2u}[y,x]^u,1\bigr);
&&S(x,y)=\bigl(1,x^{2u}y[y,x]^u\bigr).
\end{align*}

Conversely, all of the maps above define verbal solutions of the Yang--Baxter equation on any two-step nilpotent group for any values of the parameters $u,v\in\Z$.

\end{theorem}

\bigskip

%%%%%%%%%%%%%%%%%%%%%%%%%%%%%%%%%%%%%%%%%%%%%

\begin{ack}
This work is supported by the Theoretical Physics and Mathematics
Advancement Foundation BASIS No 23-7-2-14-1.
The first author was supported by the state contract of the Sobolev
Institute of Mathematics, SB RAS (No. I.1.5, project FWNF-2022-0009).

The authors thank  V. Gubarev for useful discussions and suggestions. 

\end{ack}

%\addcontentsline{toc}{section}{References}

\end{document}